\documentclass[]{amsart}

\usepackage{amsmath}
\usepackage{amstext}
\usepackage{amsthm}
\usepackage{amsfonts}
\usepackage{amssymb}
\usepackage{mathrsfs}
\usepackage{color}
\usepackage{comment}
\usepackage{enumitem}
\usepackage{mathtools}
\usepackage{nicefrac}

\usepackage[numbers,sort&compress]{natbib}

\DeclarePairedDelimiter\abs{\lvert}{\rvert}
\DeclarePairedDelimiter\norm{\lVert}{\rVert}

\makeatletter
\let\oldabs\abs
\def\abs{\@ifstar{\oldabs}{\oldabs*}}
\let\oldnorm\norm
\def\norm{\@ifstar{\oldnorm}{\oldnorm*}}
\makeatother

\makeatletter
\newcommand*{\rom}[1]{\expandafter\@slowromancap\romannumeral #1@}
\makeatother

\newtheorem{crl}{Corollary}[section]
\newtheorem{prop}[crl]{Proposition}
\newtheorem{thm}[crl]{Theorem}
\newtheorem{lem}[crl]{Lemma}
\newtheorem{rmk}[crl]{Remark}

\theoremstyle{definition}
\newtheorem{defn}[crl]{Definition}

\newcommand{\rank}{\operatorname{rank}}
\newcommand{\Op}{\operatorname{Op}}
\newcommand{\Tr}{\operatorname{Tr}}
\newcommand{\hl}{}

\numberwithin{equation}{section}

\allowdisplaybreaks

\title[Wavelet Techniques in Spectral Theory]{Estimates for Schur Multipliers and Double Operator Integrals --- A Wavelet Approach}
\author[E.~McDonald]{Edward McDonald}
\address{School of Mathematics and Statistics\\ University of New South Wales\\ Kensington, NSW 2052\\ Australia}
\email{edward.mcdonald@unsw.edu.au}
\author[T.~T.~Scheckter]{Thomas Tzvi Scheckter}
\address{School of Mathematics and Statistics\\ University of New South Wales\\ Kensington, NSW 2052\\ Australia}
\email{t.scheckter@unsw.edu.au}
\author[F.~A.~Sukochev]{Fedor Sukochev}
\address{School of Mathematics and Statistics\\ University of New South Wales\\ Kensington, NSW 2052\\ Australia}
\email{f.sukochev@unsw.edu.au}

\begin{document}

\begin{abstract}
	We discuss the work of Birman and Solomyak on the singular numbers of integral
	operators from the point of view of modern approximation theory, 
	in particular with the use of wavelet techniques. 
	We are able to provide a simple proof of norm estimates
	for integral operators with kernel in 
	$B^{\nicefrac{1}{p}-\nicefrac{1}{2}}_{p,p}(\mathbb R,L_2(\mathbb R))$.
	This recovers, extends and sheds new light on a theorem of Birman and Solomyak.
	We also use these techniques to provide a simple proof of
	Schur multiplier bounds for double operator integrals, with bounded symbol
	in 
	$B^{\nicefrac{1}{p}-\nicefrac{1}{2}}_{\frac{2p}{2-p},p}
	(\mathbb R,L_\infty(\mathbb R))$,
	which extends Birman and Solomyak's result to symbols without compact
	domain.
\end{abstract}

\maketitle

\section{Introduction}

Through their study of differentiability of functions of Hermitian operators,
Dalecki\u{\i} and Krein
\cite{Daleckii:Krein:1951, Daleckii:Krein:1956}
were
led to initiate the theory of double operator integrals, a powerful tool built
from spectral theory, which has now found itself deeply embedded in
pseudodifferential operator theory, harmonic analysis, and mathematical physics.
Although not the first works studying these methods, the series of papers
\cite{Birman:Solomyak:1966, Birman:Solomyak:1967, Birman:Solomyak:1973}
by Birman and Solomyak undoubtedly form the foundations for the further study of
double operator integrals (from here referred to as DOIs).
These papers were revolutionary at the time, and are still fundamental, but
suffer from technically challenging proofs, which are difficult to further build
upon.
We wish to rectify this, combining advances in nonlinear approximation and
wavelet analysis in order to simplify the proofs, and strengthen the
beautiful results which lie at centre of Birman and Solomyak's theory.

Here we consider a question which forms the basis for further study of DOIs ---
what norm bounds does a DOI admit for a given symbol?
Following the footsteps of Birman and Solomyak, we are able to provide a
simple proof technique, building upon the machinery afforded to us
through wavelet analysis and non-linear 
approximation theory, as developed in the '80s and '90s.

Unfortunately, so far as the authors are aware, the use of these 
techniques to study Schur multipliers has been very limited.
However, there are some clear examples, such as the early work of Peng
\cite{Peng:1993}, who used wavelet bases to find Schatten--von Neumann
class norm estimates
for certain integral operators.
Only recently, these techniques were adapted in order to prove new 
Lipschitz estimates in Schatten--von Neumann ideals \cite{MS:2020}

We mark three contributions to the theory of DOIs.
The first contribution is to extend Schatten--von Neumann 
class and Schur multiplier
estimates to DOIs having symbol with non-compact support.
Secondly, we loosen the smoothness restrictions on the symbol, and finally, we
provide clear and succinct proofs of these results, building a new methodology.

It is only through powerful wavelet analysis and approximation techniques that
we are able to both simplify and extend the results of Birman and
Solomyak, and we hope that in turn, this new approach will allow for working
mathematicians and physicists, not already familiar with the technicalities
of the theory, to develop new applications for DOIs.

Finally, the authors would like to thank Professor Rozenblyum for his many
helpful comments.

\subsection{Overview of the results}
Recall that if $T\in\mathcal K(\mathcal H)$ is a compact linear operator on a
Hilbert space
$\mathcal{H}$, then the singular value sequence $\mu(T)$ is defined by
\[
    \mu(T) = (\mu(n,T))_{n=0}^\infty,\qquad
	\mu(n,T) = \inf\{\|T-R\|_\infty:\rank(R)\leq n\}
\]
where $\norm{\cdot}_\infty$ is the operator norm.
For $0 < p < \infty$, an operator $T$ is said to belong to the
Schatten--von Neumann $\mathcal{L}_p(\mathcal{H})$-class if
$\|T\|_p = \left(\sum_{n=0}^\infty \mu(n,T)^p\right)^{\frac{1}{p}} < \infty$.
For $1\leq p < \infty$,
$\mathcal{L}_p$ is a Banach ideal of the algebra of all bounded linear 
endomorphisms of $\mathcal{H}$ and is a quasi-Banach ideal
for $0 < p < 1$.
Similarly, $\mathcal{L}_{p,q}(\mathcal{H})$ is defined as the space
of operators $T$ with singular value sequence belonging to the Lorentz sequence
space $\ell_{p,q}$ (see, for example, \cite{LT2} for the definition).

For a square-integrable function $k \in L_2(\mathbb{R}^2)$, denote by $\Op(k)$
the corresponding integral operator on $L_2(\mathbb{R})$.
That is,
\[
	\Op(k)(f)(y)=\int_{\mathbb R}k(x,y)f(x)dx,
	\quad f \in L_2(\mathbb{R}),\, y \in \mathbb{R}.
\]
Here, and throughout, we understand a function in some vector-valued Besov
class, $k\in B^s_{p,q}(\mathbb R;L_r(\mathbb R))$, as being a function over
$\mathbb R^2$.
Consider a function $k:\mathbb R^2\to\mathbb R$.
Clearly, for any $x\in\mathbb R$, $k(x,\cdot)$ is a function over $\mathbb R$,
and in this sense we may identify any function over $\mathbb R^2$ with a
function-valued function.
Conversely, let $k\in B^s_{p,q}(\mathbb R;L_r(\mathbb R))$.
Then there exists a function $k_s:\mathbb R\to\mathbb R\in L^r(\mathbb R)$, for
each $s\in\mathbb R$, such that $k(s)=g_s$.
Then, for each $t\in\mathbb R$, we may treat $k$ as a function over
$\mathbb R^2$ by identifying $k$ with $k(s,t)=k_s(t)$.
We will make us of this identification throughout.

One of the primary goals of Birman and Solomyak's 1977 survey
\cite{Birman:Solomyak:1977} is to give sufficient conditions
on $k$ such that $\Op(k) \in \mathcal{L}_p(L_2(\mathbb{R}))$.
This is achieved by highly technical results of piecewise polynomial
approximation.
The essential idea is to find conditions on $k$ such that there exists
a sequence of functions of the form
$(t,s)\mapsto \sum_{j=1}^n \xi_j(t)\eta_j(s)$
which approximate $k$ sufficiently quickly.
Numerous results of this nature were obtained.
Our aim here is to shed new light on these ideas through modern techniques of
non-linear approximation theory and wavelet analysis.
Notation and definitions concerning Besov spaces will be given in the next
section.

The following is a special case of \cite[Proposition 2.1]{Birman:Solomyak:1977}:
\begin{thm}\label{BS_singular_value_theorem_special_case}
	Let $I$ be an open bounded interval in $\mathbb{R}$,
	and let $0 < p \leq 2.$ Let $k\in L_2(\mathbb{R}^2)$
	belong to the vector-valued Besov class $B^{\alpha}_{q,q}(I,L_2(I))$, where
    \begin{equation*}
        \frac{1}{p}= \alpha+\frac{1}{2},\quad q \geq \max
		\left\{2,\nicefrac{1}{\alpha}\right\}.
    \end{equation*}
    Then $\Op(k) \in \mathcal{L}_{p,\infty}(L_2(\mathbb{R})).$
\end{thm}

\begin{rmk}
    The original statement of Birman and Solomyak applied to integral operators
	on $L_2(I^d)$ for $d\geq 1$.
	Here we restrict to the one-dimensional case for simplicity.
    Moreover, the space $B^{\alpha}_{p,p}(I,L_2(I))$ was instead given as a
	Sobolev--Slobodetskii space, which is smaller than the Besov space when
	$\alpha$ is an integer.
\end{rmk}

Throughout the paper, let us write that $a\lesssim b$ if there exists a
constant $C$ such that $a\le C b$, and $a\approx b$ if there exists a constant
$C$ such that $C^{-1}b\le a\le Cb$.
This constant may change from line to line.
If, for some given variables $p,q$, the constant is dependent upon those
variables, we may write $\lesssim_{p,q}$, or $\approx_{p,q}$.

In this paper we give a new perspective, and strengthen,
Theorem~\ref{BS_singular_value_theorem_special_case}, by proving that if
$k \in L_2(\mathbb{R}^2)\cap 
\dot{B}^{\nicefrac{1}{p}-\nicefrac{1}{2}}_{p,p}
(\mathbb{R},L_2(\mathbb{R}))$ then $\Op(k)$
belongs to $\mathcal{L}_p$.
The following is our first key result, and is proved in Section \ref{singular_values_section}.

\begin{thm}\label{intro_main_schatten_result}
    Let $0 < p \leq 2$, and let $k\in L_2(\mathbb{R}^2)$ belong to the 
	vector-valued homogeneous Besov class
	$\dot{B}^{\nicefrac{1}{p}-\nicefrac{1}{2}}_{p,p}
	(\mathbb{R},L_2(\mathbb{R}))$.
	Then $\Op(k)$ belongs to $\mathcal{L}_p$, and
    \begin{equation*}
        \|\Op(k)\|_p \lesssim 
		\|k\|_2+|k|_{\dot{B}^{\nicefrac{1}{p}-
		\nicefrac{1}{2}}_{p,p}(\mathbb{R},L_2(\mathbb{R}))}.
    \end{equation*}
\end{thm}

Another important contribution of Birman and Solomyak concerns estimates for
Schur multipliers. 
For $0 < p \leq 2$, the $\mathfrak M_p$-norm of a bounded measurable function
$k$ on $I^2$, where $I\subseteq \mathbb{R}$, is given by
\begin{equation*}
    \|k\|_{\mathfrak M_p(I^2)} =
	\sup_{\|\Op(\phi)\|_{\mathcal{L}_p(L_2(I))\leq 1}} 
	\|\Op(\phi k)\|_{\mathcal{L}_p(L_2(I))}. 
\end{equation*}
We review material concerning Schur multipliers in Section~\ref{schur_section} 
below.

For $0 < p < 2$, we denote $p^\flat = \frac{2p}{2-p}$.

\begin{thm}[{\cite[Theorem~9.2]{Birman:Solomyak:1977},%
	See also \cite[Theorems~3,9]{Birman:Solomyak:1966}}]
	For any index $0<p\le 1$, let $\alpha>\nicefrac{1}{p^\flat}$.
	Over any bounded open interval $I$, and any function $k$ on $I\times I$,
	we have the estimate
	\[
		\norm{k}_{\mathfrak M_p(I^2)}
		\lesssim_I
		\norm{k}_{B^\alpha_{2,\infty}(I,L_\infty(I))}.
	\]
\end{thm}

\begin{rmk}
	\begin{enumerate}
		\item
			The dependence of the constant on the choice of interval $I$
			prevents the proof technique from being adapted to functions
			without compact support, as limiting arguments cannot be applied.
			Our proof technique then provides a substantial improvement by
			removing this limitation, such that we may consider functions
			without compact support.
		\item
	As before, we state here only the $1$-dimensional variant of
	\cite[Theorem~9.2]{Birman:Solomyak:1977}.
	Note that in Birman and Solomyak's paper, $B^\alpha_{2,\infty}(I)$ is called
	a Nikolskii--Besov space, and is denoted by $H^\alpha_2(I)$, and which
	should not be confused for the Bessel potential space $B^{\alpha}_{2,2}(I)$
	which is typically also denoted by $H^\alpha_2(I)$.
		\item
			Our approach may be further extended to study Besov spaces over
			$\mathbb R^d$, for any dimension $d$, however here we only
			cover the $1$-dimensional variant for simplicity of exposition,
			and to avoid unnecessary technicalities.
	\end{enumerate}
\end{rmk}

The following is our second key result, and 
is proved below in Section \ref{schur_section}
(see Theorem~\ref{thm:schur}).

\begin{thm}
	For any index $p\in(0,2)$, if
	$k\in \dot{B}^{\nicefrac{1}{p^\flat}}_{p^\flat,p}
	(\mathbb R,L_\infty(\mathbb R))
	\cap L_\infty(\mathbb R^2)$, then $k$ is an $\mathcal L_p$-Schur multiplier,
	with the quasinorm estimate
	\[
		\norm{k}_{\mathfrak M_p(\mathbb{R}^2)}\lesssim_p
			|k|_{\dot{B}^{\nicefrac{1}{p^\flat}}_{%
			p^\flat,p}(\mathbb R,L_\infty(\mathbb R))}
			+
			\norm{k}_{\infty}
		.
	\]
\end{thm}

\section{Wavelets and Vector-Valued Besov Spaces}\label{besov_section}

Here we approach the construction of vector-valued Besov spaces through Meyer's
wavelet characterisation.
Recall that an orthonormal wavelet is a function $\varphi\in L_2(\mathbb{R})$
such that the family
\[
	\left\{ \varphi_{j,k}(x)=2^{\nicefrac{k}{2}}
	\varphi(2^kx-j) :j,k\in\mathbb Z\right\}
\]
of translations and dilations of $\varphi$ forms an orthonormal basis of
$L_2(\mathbb{R}).$
We call this family of functions the \textit{wavelet system}.

A celebrated theorem of Daubechies \cite{Daubechies:1988} states that for every $N\geq 0$ there exists
an $N$-times continuously differentiable compactly supported wavelet. 
If $f$ is a locally integrable function and $\varphi$ is a compactly supported continuous wavelet then the wavelet coefficient $\langle \psi_{j,k},f\rangle$ is well-defined.
{\hl 
Recall that the homogeneous Besov space $\dot{B}^s_{p,q}(\mathbb{R})$ for $s \in \mathbb{R}$, $0 < p,q\leq \infty$ may be described as the class of tempered distributions $f\in \mathcal{S}'(\mathbb{R})$ such that
\[
    \left(2^{ns}\|\Delta_n f\|_{L_p(\mathbb{R})})_{n\in \mathbb{Z}} \in \ell_q(\mathbb{Z}\right).
\]
Here, $\{\Delta_n\}_{n\in \mathbb{Z}}$ is a Littlewood-Paley decomposition of $\mathbb{R}$, see e.g. \cite[Section 2.2.1]{Grafakos:2014} or \cite[Section 2.4]{Sawano:2018} for details. The $\ell_{q}$-quasi-norm
of the above sequence is $|f|_{\dot{B}^s_{p,q}}$, the homogeneous Besov semi-quasi-norm of $f$. Observe that as polynomial
functions have Fourier transform supported at $\{0\}$,
it follows by definition that all polynomials belong to $\dot{B}^s_{p,q}(\mathbb{R})$ with vanishing semi-quasi-norm.

Homogeneous Besov spaces admit a simple characterisations in terms of wavelet coefficients. The first such characterisation is due to Meyer \cite{Meyer:1987}, who proved
that if $\psi$ is a Schwartz class wavelet then a distribution $f\in \mathcal{S}'(\mathbb{R})$ belongs to $\dot{B}^s_{p,q}(\mathbb{R})$ if and only if
\[
    \left(2^{j\left(s+\nicefrac{1}{2}-\nicefrac{1}{p}\right)}
	\|\left(\langle f,\psi_{j,k}\rangle\right)_{k\in \mathbb{Z}}
	\|_{\ell_p\left(\mathbb{Z}\right)}\right)_{j\in \mathbb{Z}} \in \ell_q\left(\mathbb{Z}\right).
\]
See also \cite[Section 2.4]{Sawano:2018}.
A Schwartz class wavelet $\psi$ necessarily has vanishing moments to all orders (that is, $\int_{\mathbb{R}} t^l\psi(t)\,dt=0$ for all $l\geq 0$) \cite[Chapter 3, Section 7]{Meyer:1992}, and hence this characterisation
is at least consistent with the fact that all polynomials belong to $\dot{B}^{s}_{p,q}(\mathbb{R}).$

In our case we will require wavelets that are compactly supported. A compactly supported wavelet $\varphi$ can have at most a finite degree of regularity, because otherwise all moments of $\varphi$ vanish and this would imply that $\varphi=0$ \cite[Theorem 3.8]{HernandezWeiss:1996}.

The homogeneous Besov seminorm can be estimated in terms of wavelet coefficients with a finite degree of smoothness, as in Meyer \cite[Chapter 6, Section 10]{Meyer:1992}. This issue is somewhat more subtle than 
with Schwartz class wavelets. The first reason is that if $\varphi$ is a $C^N$-wavelet then the wavelet coefficient $\langle f,\varphi_{j,k}\rangle$ is not defined for all tempered distributions $f$, but instead
only for those with sufficiently high regularity. The second issue is that wavelets with a finite degree of smoothness have only a finite number of vanishing moments.
}
\begin{thm}
    Let $s \in \mathbb{R}$ and $0<p,q\leq \infty$ {\hl satisfy $s > \max\{\frac{1}{p}-1,0\}.$} Let $\varphi$ be a compactly supported $C^N$-wavelet, where $N> |s|.$
    
    {\hl A distribution $f \in \mathcal{S}'(\mathbb{R})$ belongs to the homogeneous Besov space $\dot{B}^{s}_{p,q}(\mathbb{R})$ if and only if there exists a sequence of polynomials $(P_{j,k})_{j,k\in \mathbb{Z}}$
    and constants $(c_{j,k})_{j,k\in \mathbb{Z}}$ such that
    \[
		f = \sum_{j,k\in \mathbb{Z}} c_{j,k}\varphi_{j,k}+P_{j,k},
    \]
    where the series converges with respect to topology of $\mathcal{S}'(\mathbb{R})$ and
    \begin{equation*}
        \left(2^{n\left(s+\nicefrac{1}{2}-\nicefrac{1}{p}\right)}\|\left(c_{j,k}\right)_{k\in \mathbb{Z}}\|_{\ell_p\left(\mathbb{Z}\right)}\right)_{n\in \mathbb{Z}} \in \ell_q\left(\mathbb{Z}\right).
    \end{equation*}
    The infimum of the $\ell_q(\mathbb{Z})$ quasi-norm of the above sequence over all representations of $f$ is equivalent to corresponding Besov semi-quasi-norm of $f$.}
\end{thm}
{\hl This theorem is stated in essentially the same form as \cite[pp. 201]{Meyer:1992}, with modifications for the case $p<1$ as in \cite[Theorem 3.7.7]{Cohen:2003}. The case with $s \leq \max\{\frac{1}{p}-1,0\}$ is discussed in \cite[Remark 3.7.5]{Cohen:2003} but is not relevant for our present applications.}

{\hl  If $E$ is a Banach space, then the space of $E$-valued tempered distributions $\mathcal{S}'(\mathbb{R},E)$ is defined
as the space of all continuous linear maps $T:\mathcal{S}(\mathbb{R})\to E$. The space $\mathcal{S}'(\mathbb{R},E)$ is equipped with a topology of pointwise norm-convergence, see \cite[Definition 2.4.24]{HvNVW:2016}. We shall take the wavelet characterisation of Besov spaces as motivation for the following definition:}
\begin{defn}
	\label{defn:besov}
    Let $E$ be a Banach space, and let {\hl $f \in \mathcal{S}'(\mathbb{R},E)$.
    Let $s\in \mathbb{R}$ and $0<p,q\leq \infty$ satisfy $s > \max\{\frac{1}{p}-1,0\}.$ Let $\varphi$ be a compactly supported $C^N$-wavelet, where $N>|s|.$
    Say that $f$ belongs to the homogeneous Besov space $\dot{B}^{s}_{p,q}(\mathbb{R},E)$ if $f$ can be represented as
    \[
		f = \sum_{j,k\in \mathbb{Z}} c_{j,k}\varphi_{j,k}+P_{j,k},
    \]
    where the series converges with respect to the topology of $\mathcal{S}'(\mathbb{R},E)$ where $P_{j,k}$ are polynomials with coefficients in $E$, and $c_{j,k}\in E$ are such that
    \begin{equation*}
		\norm{\left(2^{n(s+\nicefrac{1}{2}-\nicefrac{1}{p})}
			\norm{\left(
				\norm{
					c_{j,k}
				}_E\right)_{k\in \mathbb{Z}}
			}_{\ell_p(\mathbb{Z})}\right)_{n\in \mathbb{Z}}
		}_{\ell_q(\mathbb{Z})} < \infty.
    \end{equation*}
    }
    {\hl The Besov semi-quasi-norm $\abs{f}_{\dot{B}^{s}_{p,q}(\mathbb{R},E)}$ of $f$ is defined to be the infimum of the above quantity over all representations of $f.$ }
\end{defn}

For our present purposes we will only
need $E = L_2(\mathbb{R})$ and $E = L_\infty(\mathbb{R})$. 

\section{$n$-term approximation of $L_2(\mathbb{R})$-valued functions}
The ideas at the heart of Birman and Solomyak's original proofs revolve around
developments in non-linear approximation theory for piecewise polynomials.

Let $(X,\norm{\cdot}_X)$ be a quasi-normed linear space. A sequence of subsets $(X_n)_{n=0}^\infty$
of $X$ will be called an \emph{approximation {\hl scheme}} for $X$ if 
\begin{enumerate}
	\item We have that $X_0=\left\{ 0 \right\}$.
	\item For all $n\geq 0$, $X_n\subseteq X_{n+1}$.
	\item For all $n\geq 0$, and any $a\in\mathbb C$, $aX_n=X_n$.
	\item There exists an integer $k>0$, such that for any
		$n\in I$, $X_n+X_n\subseteq X_{kn}$.
	\item The set $\bigcup_{n=0}^\infty X_n$ is dense in $X$.
\end{enumerate}
If $X_n+X_n=X_n$ for every $n\geq 0$, we call say that
$(X_n)_{n=0}^\infty$ is a linear approximation {\hl scheme}, and otherwise we will
call it non-linear.

\begin{defn}
	Given a quasi-normed space $(X,\norm{\cdot}_X)$,
	and an approximation {\hl scheme} $(X_n)_{n=0}^\infty$ of (possibly non-linear)
	subsets of $X$
	the $n$th \textit{approximation number} for an element
	$f\in X$ is defined by
	\[
		E_n(f)_X=\inf\limits_{g\in X_n}\norm{f-g}_X.
	\]
	Denote $E(f)_X = (E_n(f)_X)_{n=0}^\infty$ for the sequence of approximation numbers.

    For $\alpha>0$, and $q\in(0,\infty]$, define the quasi-norm
\[
	\norm{f}_{\mathcal A^\alpha_q}=
	\norm{
		E(f)_X
	}_{\ell_{\alpha^{-1},q}}.
\]
The \textit{approximation space} 
\[
\mathcal A^\alpha_q=\mathcal A^\alpha_q(X,(X_n)_{n\in\mathbb N})
\]
is the
set of all $f\in X$ such that $\norm{f}_{\mathcal A^\alpha_q}<\infty$.
\end{defn}

Let us now consider a nonlinear approximation scheme $(T_n\otimes E)_{n\ge 0}$,
for some Banach space $E$, where for each $n\in\mathbb N$,
\begin{equation}\label{n_term_wavelet_scheme}
	T_n=
	\left\{ 
		f=\sum_{l=1}^n c_l\varphi_{j_l,k_l}:
		c_l\in \mathbb{C},\ j_l,k_l\in\mathbb Z
	\right\},
\end{equation}
that is the space of all 
functions given by a linear combination of any $n$
functions in the wavelet system.

Before we may state our approximation theorem, we will also need a variant of
the discrete Hardy inequality.

\begin{lem}[{See, for example, \cite[Chapter~2, Lemma~3.4]{DeVore:Lorentz:1993}}]
	\label{lem:hardy}
	For any two
	positive, monotonically decreasing, sequences $(a_n)_{n\in\mathbb N}$, 
	$(b_n)_{n\in\mathbb N}$, of non-negative real numbers, if there exist constants
	$\mu,r,C>0$, such that
	\[
		a_n\le
		Cn^{-r}
		\left( 
			\sum\limits_{k=n}^\infty k^{r\mu-1}b_k^\mu
		\right)^{\nicefrac{1}{\mu}},
	\]
	for all $n\ge 1$, then it follows for all
	$0<q\le \infty$ and $s>r$ that
	\[
		\norm{
			(a_n)_{n\in\mathbb N}
		}_{\ell_{s^{-1},q}}
		\lesssim_{\mu,s,r,C}
		\norm{
			(b_n)_{n\in\mathbb N}
		}_{\ell_{s^{-1},q}}.
	\]
\end{lem}

The following characterises the approximation spaces for the space $L_2(\mathbb{R},L_2(\mathbb{R}))$ with respect
to the approximation scheme $(T_n\otimes L_2(\mathbb{R}))_{n=0}^\infty$ {\hl from \eqref{n_term_wavelet_scheme}}. The proof is identical to the corresponding
scalar-valued case \cite[Section 7.6]{DeVore:1998}, we supply it here for convenience.
\begin{thm}
	\label{thm:nonlinear}
	For any index $p\in (0,2)$, we have
	\[
		\dot{B}^{\nicefrac{1}{p}-\nicefrac{1}{2}}_{p,p}(\mathbb R, L_2(\mathbb R)){\hl \cap L_2(\mathbb{R}^2)}=
		\mathcal A^{\nicefrac{1}{p}-\nicefrac{1}{2}}_p
		(L_2(\mathbb R,L_2(\mathbb R)), (T_n\otimes L_2(\mathbb R)_{n\in\mathbb N})).
	\]
	
	In particular, the space
	${\hl L_2(\mathbb{R}^2)\cap}\dot{B}^{\nicefrac{1}{p}-\nicefrac{1}{2}}_{p,p}(\mathbb R, L_2(\mathbb R))$
	is then the space of all $L_2(\mathbb R)$-valued functions $f$ such that
	\[
		\left\{ 
			\norm{\langle \varphi_{j,k},f\rangle}_2
		\right\}_{j,k\in\mathbb Z}\in \ell_{p}(\mathbb Z^2),
	\]
	{\hl and
	\[
		\|f\|_{L_2(\mathbb{R}^2)} < \infty.
	\]
	}
\end{thm}
\begin{proof}
	It follows from Definiton~\ref{defn:besov} that a function
	$f\in \dot{B}^{\nicefrac{1}{p}-\nicefrac{1}{2}}_{p,p}(\mathbb R, L_2(\mathbb{R}))$
	if and only if
	\[
		\sum\limits_{j,k\in\mathbb Z}
		\norm{
			\langle 
				\varphi_{j,k},f
			\rangle
		}_2^p<\infty.
	\]
	Let $(b_n)_{n\in\mathbb N}$ denote the decreasing rearrangement of the
	sequence
	\[
		\left\{ \norm{
			\langle\varphi_{j,k},f\rangle}_2
		\right\}_{j,k\in\mathbb Z}.
	\]
	Since the spaces $\varphi_{j,k}\otimes L_2(\mathbb{R})$ are orthogonal in $L_2(\mathbb{R},L_2(\mathbb{R}))$, it follows 
	by the definition of the approximation numbers that we have
	\[
		E_n(f)=
		\inf\left\{ \norm{f-g}_{L_2(\mathbb R,{\hl L_2(\mathbb{R})})}:f\in T_n{\hl \otimes L_2(\mathbb{R})} \right\}
		=
		\left( \sum\limits_{k=n+1}^{\infty}b_k^2 \right)^{\nicefrac{1}{2}}.
	\]
	By the discrete Hardy inequality,
	it follows for any $r<2$, and $0<q\le \infty$ that
	\[
		\norm{
			\left\{ 
				\frac{E_n(f)}{(n+1)^{\nicefrac{1}{2}}}
			\right\}_{n\in\mathbb N}
		}_{\ell_{r,q}}\lesssim
		\norm{(b_n)_{n\in\mathbb N}}_{\ell_{r,q}}.
	\]
	In particular,
	\[
		\norm{f}_{
			\mathcal A_{q}^{\nicefrac{1}{r}-\nicefrac{1}{2}}
			\left( L_2(\mathbb R,L_2(\mathbb{R})), 
			(T_n\otimes L_2(\mathbb R))_{n\in\mathbb N} \right)
		}
		\lesssim\norm{\left( b_n \right)_{n\in\mathbb N}}_{\ell_{r,q}}.
	\]
	Simplifying for the case when $p=q=r$, it follows from the definition
	of the sequence $(b_n)_{n\in\mathbb N}$, and Definition~\ref{defn:besov},
	that we have
	\[
		\norm{f}_{
			\mathcal A_{p}^{\nicefrac{1}{p}-\nicefrac{1}{2}}
			\left( L_2(\mathbb R,L_2(\mathbb{R})), 
			(T_n\otimes L_2(\mathbb R))_{n\in\mathbb N} \right)
		}
		\lesssim
		{\hl |b_0|+}\left( 
			\sum\limits_{j,k}
			\norm{\langle \varphi_{j,k},f\rangle}_2^p
		\right)^{\nicefrac{1}{p}}
		\approx
		{\hl |b_0|+}|{f}|_{\dot{B}_{p,p}^{\nicefrac{1}{p}-\nicefrac{1}{2}}
		(\mathbb R,L_2(\mathbb R))}.
	\]
	Given that the sequence $(b_n)_{n\in\mathbb N}$ is monotonically decreasing,
	$b_{2n}^2\le b_k^2$, for all $0\le k<2n$.
	In turn, we have that
	\[
		nb_{2n}^2\le
		\sum\limits_{k=n}^{2n+1}b_k^2
		\le
		\sum\limits_{k=n}^{\infty}b_k^2
		=E_{n-1}(f)^2.
	\]
	As such, we have that $n^{\nicefrac{1}{2}}b_{2n}\le E_{n-1}(f)$, and so
	$(2n)^{\nicefrac{1}{2}}b_{2n}\lesssim E_{n-1}(f)$.
	The reverse inclusion then follows.

	From this {\hl and $b_0 = \|f\|_{L_2(\mathbb{R}^2)}$} it follows that
	\[
{\hl 		\|f\|_{L_2(\mathbb{R}^2)}+}|f|_{\dot{B}^{\nicefrac{1}{p}-\nicefrac{1}{2}}_{p,p}(\mathbb R,L_2(\mathbb{R}))}
		\approx
		\norm{\left( b_n \right)_{n\in\mathbb N}}_{\ell_p}
		\le
		\norm{f}_{
			\mathcal A_{p}^{\nicefrac{1}{p}-\nicefrac{1}{2}}
			\left( L_2(\mathbb R,L_2(\mathbb{R})), (T_n\otimes L_2(\mathbb{R}))_{n\in\mathbb N} \right)
		}.
	\]
	This completes the proof.
\end{proof}

\section{Approximation of Singular Values for Integral Operators}\label{singular_values_section}

Both Birman and Solomyak's initial research into DOIs
\cite{Birman:Solomyak:1966, Birman:Solomyak:1967}, and their later
investigations into integral operators \cite{Birman:Solomyak:1977}, capitalise
upon the insight that approximation theory techniques allow for one
to effectively find norm estimates for integral operators.

Observe that the nested family of subsets 
\[
    \mathcal{R}_n = \{T\in \mathcal{K}(\mathcal{H})\;:\;\rank(T)\leq n\},\quad n\geq 0
\]
is an approximation scheme for $\mathcal{K}(\mathcal{H}),$
and the singular value sequence $\mu(T)$ is precisely the sequence of approximation numbers $E(T)_{\mathcal{K}(\mathcal{H})}.$ Correspondingly,
the Schatten--von Neumann ideals can equivalently be described as approximation spaces,
\[
    \mathcal{L}_{p,q}(\mathcal{H}) = \mathcal A^{\nicefrac{1}{p}}_{q}(\mathcal{K}(\mathcal{H}),(\mathcal{R}_n)_{n=0}^\infty).
\]
%
The primary goal of Birman and Solomyak's survey \cite{Birman:Solomyak:1977}
is to estimate the singular values of integral operators $\Op(k)$ in terms of $k$.
The main difficulty is that there is no sufficiently general estimate for the operator norm of an integral operator in terms of its kernel.

However, it is well-known that the Hilbert--Schmidt norm is given by
\[
	\norm{\Op(k)}_2=\norm{k}_{L_2(\mathbb R^2)}.
\]
The $\mathcal L_2$-approximation numbers are then given by
\[
	e_n(T)=
	\inf
	\left\{ 
		\norm{T-R}_2:
		\rank(R)\le n
	\right\},
\]
for each $n\ge 0$.
Of course, these are the approximation numbers for $\mathcal L_2$, with
respect to the approximation {\hl scheme} $(\mathcal{R}_n)_{n=0}^\infty$. That is, $e_n(T) = E_n(T)_{\mathcal{L}_2(\mathcal{H})}.$

The following simple lemma is at the heart of Birman and Solomyak's results, and allows us to replace the operator
norm with the Hilbert-Schmidt norm.
\begin{lem}[{\cite[Lemma~1.3]{Birman:Solomyak:1977}}]
	\label{lem:mu:estimate}
	For any compact operator $T$, and every $n\ge 1$, we have that
	\[
		\mu(2n,T)\le n^{-\nicefrac{1}{2}}e(n,T).
	\]
\end{lem}
\begin{proof}
	For any operator $R$ with $\rank(R)\le n$,
	\[
		\mu(2n,T)\le
		\mu(n,T-R)+\mu(n,R)
		=
		\mu(n,T-R)\le
		n^{-\nicefrac{1}{2}}
		\norm{T-R}_2.
	\]
	The result follows by taking the infimum over all such $R$.
\end{proof}

It is remarkable that this simple lemma is sufficiently sharp for estimates
of $\mathcal L_{p,q}$-norms.
\begin{prop}
	\label{prop:lpq}
	For any compact operator $T$, and indices
	$0<p<2$, $0<q\le\infty$, we have that
	\[
		\norm{T}_{\mathcal L_{p,q}}
		\approx
		\norm{
			\left( 
				\frac{e(n,T)}{(n+1)^{\nicefrac{1}{2}}}
			\right)_{n\in\mathbb N}
		}_{\ell_{p,q}}
		=
		\norm{
			(e(n,T))_{n\in\mathbb N}
		}_{\ell_{( \nicefrac{1}{p}-\nicefrac{1}{2} )^{-1},q}}.
	\]
\end{prop}

\begin{proof}
	To start, note that
	\[
		e(n,T)=
		\left( 
			\sum\limits_{k=n+1}^\infty\mu(k,T)^2
		\right)^{\nicefrac{1}{2}},
	\]
	for any $n\ge 0$.
	
	It then follows by application of the discrete Hardy inequality,
	Lemma~\ref{lem:hardy}, when $\mu=2$ and $r=\nicefrac{1}{2}$, that
	\[
		\norm{
			\left(
				\frac{e(n,T)}{(n+1)^{\nicefrac{1}{2}}}
			\right)_{n\geq 0}
		}_{\ell_{p,q}}
		\lesssim
		\norm{
			\left( \mu(n,T) \right)_{n\ge 0}
		}_{\ell_{p,q}},
	\]
	for all $p\in(0,2)$, and all $0<q\le\infty$.

	By application of Lemma~\ref{lem:mu:estimate}, we recover quasi-norm
	equivalence
	for all $p\in(0,2)$, and all $0<q\le\infty$
	\[
		\norm{
			\left( 
				\frac{e(n,T)}{(n+1)^{\nicefrac{1}{2}}}
			\right)_{n\geq 0}
		}_{\ell_{p,q}}
		\approx
		\norm{
			\left( \mu(n,T) \right)_{n\ge 0}
		}_{\ell_{p,q}},
	\]
	which is equivalent to the stated result.
\end{proof}

The following result recovers many of the estimates of Birman and Solomyak
\cite[Theorem~2.4]{Birman:Solomyak:1977}, and Pietsch
\cite{Pietsch:1, Pietsch:2, Pietsch:3}.

Let us note how short the proof of the following result is, in contrast to
the other approaches just noted.

\begin{thm}\label{main_lpq_result}
	For any index $p\in(0,2)$, we have the continuous embedding
	\[
		\Op:\dot{B}^{\nicefrac{1}{p}-\nicefrac{1}{2}}_{p,p}(\mathbb R,L_2(\mathbb R)) {\hl \cap L_2(\mathbb{R}^2)}
		\to\mathcal L_p(L_2(\mathbb R)).
	\]
\end{thm}

\begin{proof}
	Let $\Sigma=\left( \Sigma_n \right)_{n\in\mathbb N}$ be the nested family
	of subsets of $L_2(\mathbb R^2)$, defined for each $n\ge 0$ by
	\[
		\Sigma_n=
		\left\{ 
			k(t,s)=\sum\limits_{j=1}^n x_j(t)y_j(s):
			x_j,y_j\in L_2(\mathbb R)
		\right\}.
	\]
	That is to say that $\Sigma_n$ consists of linear combinations of
	at most $n$ elementary tensors of functions in $L_2(\mathbb R)$.
	Under $\Op$, $\Sigma_n$ is sent
	precisely to the space $\mathcal{R}_n$ of linear operators of rank at most $n$.

	It is then immediate from Proposition~\ref{prop:lpq} that for any fixed
	index $0<q\le \infty$, we may restrict the isometry
	\[
		\Op:L_2(\mathbb R^2)\to \mathcal L_2(L_2(\mathbb R))
	\]
	to the isomorphism
	\[
		\Op:
		\mathcal A^{\nicefrac{1}{p}-\nicefrac{1}{2}}_q(L_2(\mathbb R^2),\Sigma)
		\cong
		\mathcal L_{p,q}(L_2(\mathbb R)).
	\]
	Observe that for every $n\geq 0$ we have
	\[
        T_n\otimes L_2(\mathbb{R})\subset \Sigma_n.
	\]
	It is therefore immediate that we have a continuous inclusion
	\[
		\mathcal A^{\nicefrac{1}{p}-\nicefrac{1}{2}}_q
		(L_2(\mathbb R^2),(T_n\otimes L_2(\mathbb{R}))_{n\in\mathbb N})
		\hookrightarrow
		\mathcal A^{\nicefrac{1}{p}-\nicefrac{1}{2}}_q
		(L_2(\mathbb R^2),\Sigma),
	\]
	however we know from Theorem~\ref{thm:nonlinear} that
	\[
		\dot{B}^{\nicefrac{1}{p}-\nicefrac{1}{2}}_{p,p}
		(\mathbb R,L_2(\mathbb R)){\hl \cap L_2(\mathbb{R}^2)}
		=
		\mathcal A^{\nicefrac{1}{p}-\nicefrac{1}{2}}
		(L_2(\mathbb R^2),(T_n\otimes L_2(\mathbb{R})_{n=0}^\infty)).
	\]
	As such,
	\[
		\Op:\dot{B}^{\nicefrac{1}{p}-\nicefrac{1}{2}}_{p,p}(\mathbb R,L_2(\mathbb R)){\hl \cap L_2(\mathbb{R}^2)}
		\to\mathcal L_p(L_2(\mathbb R)).
	\]
\end{proof}
{\hl This concludes the proof of Theorem \ref{intro_main_schatten_result}, which is simply a restatement of the above result.}

\begin{rmk}
	To conclude this section, let us note that Proposition~\ref{main_lpq_result}
	gives a sufficient condition for an integral operator to be trace class. Namely,
	if $k \in B^{\frac{1}{2}}_{1,1}(\mathbb{R},L_2(\mathbb{R}))$ then $\Op(k)\in \mathcal{L}_1.$
	Furthermore the operator trace of $\Op(k)$ is given by the formula
	\begin{equation*}
        \Tr(\Op(k)) = \int_{\mathbb{R}} k(x,x)\,dx
	\end{equation*}
	provided that one understands the restriction of $k$ to the diagonal $\{(x,x)\;:\;\;x\in \mathbb{R}\}$
	in the correct sense. For further details see \cite{Brislawn:1988}.
\end{rmk}

\section{Double Operator Integrals and Schur Multipliers}\label{schur_section}

Our final task is to show that if a DOI has symbol in some vector-valued Besov
class, then it is an $L_p$-Schur multiplier.
This is a difficult task in Birman and Solomyak's original papers
\cite{Birman:Solomyak:1966,
Birman:Solomyak:1967,
Birman:Solomyak:1973,
Birman:Solomyak:1977}
, which is
achieved through a careful argument, which loosely shows that by replacing
multiplication by an $L_2$-operator with a weighted measure, norm-estimates
for operators with non-smooth symbol may be replaced by those for some
corresponding smooth symbol.
Here we are able to eschew all of these measure-theoretic difficulties by
applying the wavelet characterisation of wavelet vector-valued Besov classes.
The result is then not only more easily attainable, but we remove the need
for the symbol to be compactly supported.
For background on Schur multipliers, see the survey of Aleksandrov and
Peller \cite{Aleksandrov:Peller:2002}.
Recall that for $0 < p \leq 2$, a bounded Borel measurable function $k$ on $\mathbb{R}^2$ is called an $\mathcal{L}_p$-Schur multiplier, $k\in\mathfrak M_p$ if there exists
a constant $C$ such that
\begin{equation*}
    \|\Op(k\phi)\|_p \leq C\|\Op(\phi)\|_p
\end{equation*}
for all $\phi\in L_2(\mathbb{R}^2)$ such that $\Op(\phi)\in \mathcal{L}_p(L_2(\mathbb{R})).$
The $\mathfrak M_p$-quasinorm of $k$ is defined as
\begin{equation}\label{mp_norm_def}
    \|k\|_{\mathfrak M_p} = \sup_{\|\Op(\phi)\|_p\leq 1}\|\Op(k\phi)\|_p.
\end{equation}
It follows directly from the properties of $\mathcal{L}_p$ that for $0 < p < 1$, $\mathfrak M_p$ is a quasi-Banach space and
obeys the $p$-triangle inequality
\begin{equation*}
    \left\|\sum_{j=0}^\infty k_j\right\|_{\mathfrak M_p}^p \leq \sum_{j=0}^\infty \|k_j\|_{\mathfrak M_p}^p.
\end{equation*}
For $1\leq p \leq 2$, $\mathfrak M_p$ is a Banach space, and $\mathfrak M_2 = L_\infty(\mathbb{R}^2).$

Let us note some features of the classes of Schur multipliers.
It follows directly from the properties of $\mathcal{L}_p$ that for $0 < p < 1$, $\mathfrak M_p$ is a quasi-Banach space and
obeys the $p$-triangle inequality
\begin{equation*}
    \left\|\sum_{j=0}^\infty k_j\right\|_{\mathfrak M_p}^p \leq \sum_{j=0}^\infty \|k_j\|_{\mathfrak M_p}^p.
\end{equation*}
For $1\leq p \leq 2$, $\mathfrak M_p$ is a Banach space, and $\mathfrak M_2 = L_\infty(\mathbb{R}^2).$

An important feature of the $\mathfrak M_p$-quasinorm when $0 < p \leq 1$
is that it suffices to compute the supremum in \eqref{mp_norm_def} over rank one operators. The following fundamental fact is \cite[Theorems 8.1 and 8.2]{Birman:Solomyak:1977}:
\begin{equation}\label{rank_one_suffices}
    \norm{k}_{\mathfrak M_p}
    =
    \sup\limits_{\xi,\eta}
    \norm{\Op\left( k\left( \xi\otimes\eta \right) \right)}_{\mathcal L_p},
\end{equation}
where the supremum is taken over all $\xi,\eta$ in the unit ball of
$L_2(\mathbb R)$. This is an immediate consequence of the Schmidt decomposition of a compact operator and the $p$-triangle inequality, see \cite[Lemma 2.2.1]{MS:2020} for a proof of the corresponding matrix case.

The notion of an $\mathcal{L}_p$-Schur multiplier is related to the idea of a double operator integral in the following sense. Let $M_x$
denote the unbounded self-adjoint operator on $L_2(\mathbb{R})$ of multiplication by the coordinate variable $M_x\xi(t) := t\xi(t)$. Then
\begin{equation*}
    T^{M_x,M_x}_{k}(\Op(\phi)) = \Op(k\phi),\quad \phi\in L_2(\mathbb{R}^2).
\end{equation*}
It follows that for all $0 < p \leq 2$ we have
\begin{equation*}
    \|k\|_{\mathfrak M_p} = \|T^{M_x,M_x}_{k}\|_{\mathcal{L}_p\to \mathcal{L}_p}.
\end{equation*}
Similarly, if $A$ and $B$ are unbounded self-adjoint operators on a Hilbert space $\mathcal{H}$ 
with the same spectral type as $M_x$, then
\begin{equation*}
    \|T^{A,B}_{k}\|_{\mathcal{L}_p\to \mathcal{L}_p} = \|k\|_{\mathfrak M_p}.
\end{equation*}
For more details on the relationship between Schur multipliers and double operator integrals, see \cite[Section 1.4]{Birman:Solomyak:2003}.

Recall that for any $p\in(0,2)$, we let $p^\flat=2p(2-p)^{-1}$, in order that
$\nicefrac{1}{p^\flat}=\nicefrac{1}{p}-\nicefrac{1}{2}$.
Our first result is a continuous extension of
\cite[Theorem~3.7]{Aleksandrov:Peller:2002}, which has essentially the same proof.

\begin{thm}
	\label{thm:decomp}
	Fix any function $k\in L_\infty(\mathbb R^2)$, and any index
	$0<p\le 1$.
	Then for any partition of $\mathbb R$ into intervals $I_j=[t_j,t_{j+1})$,
	for $j\in\mathbb Z$,
	let
	\[
		k_j(t,s)=\chi_{I_j}(t)k(t,s),
	\]
	for all $t,s\in\mathbb R$, $j\in\mathbb Z$.
	The following estimate then holds,
	\[
		\norm{k}_{\mathfrak M_p}\le
		\norm{
			\left( \norm{k_j}_{\mathfrak M_p} \right)_{j\in\mathbb Z}
		}_{\ell_{p^\flat}}.
	\]
\end{thm}

\begin{proof}
	Here we use \eqref{rank_one_suffices} to compute the $\mathfrak M_p$-multiplier quasinorm. 
	Let $\xi,\eta \in L_2(\mathbb{R}).$
	It follows from the decomposition $k=\sum_{j\in\mathbb Z}k_j$ and the
	triangle inequality that
	\[
		\norm{
			\Op\left( k\left( \xi\otimes\eta \right) \right)
		}^p_{\mathcal L_p}\le
		\sum\limits_{j\in\mathbb Z}
		\norm{
			\Op\left(k_j\left( \xi\otimes\eta \right)\right)
		}^p_{\mathcal L_p}.
	\]
	It is clear that for any $j\in\mathbb Z$,
	$k_j(\xi\otimes \eta)=k_j(\chi_{I_j}\xi\otimes\eta)$, such that
	\[
		\norm{
			\Op\left( k\left( \xi\otimes\eta \right) \right)
		}_{\mathcal L_p}^p
		\le\sum\limits_{j\in\mathbb Z}
		\norm{k_j}^p_{\mathfrak M_p}\norm{\chi_{I_j}\xi}^p_2\norm{\eta}^p_2.
	\]
	Applying H\"older's inequality, to the right hand side, we see that
	\[
		\norm{\Op\left( k\left( \xi\otimes\eta \right) \right)}_{\mathcal L_p}^p
		\le
		\norm{\left( \norm{k_j}_{\mathfrak M_p}
		\right)_{j\in\mathbb Z}}_{\ell_{p^\flat}}
		\norm{\left( \norm{\chi_{I_j}\xi}_2 \right)_{j\in\mathbb Z}}_{\ell_2}
		\norm{\eta}_2.
	\]
	However, the intervals $I_j$ are disjoint, such that 
	$(\chi_{I_j}\xi)_{j\in\mathbb Z}$ is a sequence of functions orthogonal
	in $L_2(\mathbb R)$, which in turn gives us that
	$\norm{(\norm{\chi_{I_j}\xi}_2)_{j\in\mathbb Z}}_2=\norm{\xi}_2$.
	The estimate
	\[
		\norm{\Op\left( k\left( \xi\otimes\eta \right) \right)}_{\mathcal L_p}^p
		\le
		\norm{\left( \norm{k_j}_{\mathfrak M_p}
		\right)_{j\in\mathbb Z}}_{\ell_{p^\flat}}
		\norm{\xi}_2
		\norm{\eta}_2
	\]
	then gives the result once we take the supremum over all $\xi$ and $\eta$
	in the unit ball of $L_2(\mathbb R)$.
\end{proof}

\begin{crl}
	For any sequence $(k_j)_{j\in\mathbb Z}\subseteq L_\infty(\mathbb R)$,
	let $\psi$ be a compactly supported, bounded function on $\mathbb R$,
	and define
	\[
		k(t,s)=\sum\limits_{j\in\mathbb Z}\psi(t-j)k_j(s).
	\]
	The following estimate then holds,
	\[
		\norm{k}_{\mathfrak M_p}
		\lesssim_{\psi}
		\norm{
			\left( \norm{k_j}_\infty \right)_{j\in\mathbb Z}
		}_{\ell_{p^\flat}}.
	\]
\end{crl}

\begin{proof}
	For disjointly supported functions
	$\left( t\mapsto \psi(t-j) \right)_{j\in\mathbb Z}$, the result follows
	immediately from Theorem~\ref{thm:decomp}.
	If not, we may choose some sufficiently large $N>0$, such that
	$\left( t\mapsto \psi(t-Nj) \right)_{j\in\mathbb Z}$, is a disjointly
	supported sequence of functions.
	Then let $k=\sum_{l=0}^{N-1}k^{(l)}$, where
	\[
		k^{(l)}(t,s)=
		\sum\limits_{j\in\mathbb Z}
		\psi(t-l-Nj)k_j(s),
	\]
	for all $s,t\in\mathbb R$.
	The result then follows from the individual bounds on each $k^{(l)}$.
\end{proof}

Let us now return to wavelet decompositions.
For each $j,l\in\mathbb Z$, and any function $k$ on $\mathbb R^2$,
define
\[k_{j,l}(s)=\int_{\mathbb R}\varphi_{j,k}(t)k(t,s)dt,\]
for all $s\in\mathbb R$.
The next corollary then follows immediately by rescaling, as
$\varphi$ is a compactly supported wavelet, where
$\varphi_{j,k}(t)=2^{\nicefrac{j}{2}}\varphi(2^jt-k)$.

\begin{crl}
	\label{crl:schur:est}
	For any bounded function
	$k\in \dot{B}^{\nicefrac{1}{p^\flat}}_{p^\flat,p}(\mathbb R,L_\infty(\mathbb R))$,
	and any $j\in\mathbb Z$, let
	\[
		k_j(t,s)=\sum\limits_{l\in\mathbb Z}
		\varphi_{j,l}(t)k_{j,l}(s),
	\]
	for all $s,t\in\mathbb R$.
	
	We then have that
	\[
		\norm{k_j}_{\mathfrak M_p}
		\lesssim_\varphi
		2^{\nicefrac{j}{2}}
		\norm{
			\left( 
				\norm{k_{j,l}}_\infty
			\right)_{l\in\mathbb Z}
		}_{\ell_{p^\flat}},
	\]
	for all $j\in\mathbb Z$.
\end{crl}

Before we proceed, let us note that for a function
$k\in \dot{B}^s_{p,q}(\mathbb R,L_\infty(\mathbb R))$, it need not be the case
that the decomposition $k=\sum_{j\in\mathbb Z}k_j$ holds
(compare to the scalar-valued failure of such a decomposition, as detailed
in \cite[Chapter~3, Proposition~4]{Meyer:1992}).
Despite this, it follows from an obvious modification of
\cite[Lemma~4.1.4]{MS:2020} that if $k\in L_\infty(\mathbb R^2)$, then there
exists a function $c\in L_\infty(\mathbb R)$, such that
\[
	k(t,s)=c(t)+\sum\limits_{j\in\mathbb Z}\left(k_j(t,s)-k_j(t,0)\right),
\]
for all $t,s\in \mathbb R^2$, and such that
\[
	\norm{c}_\infty\lesssim
	\norm{k}_\infty+
	|k|_{\dot{B}^s_{p,q}(\mathbb R, L_\infty(\mathbb R))}.
\]
Finally, we may state our strengthening of
\cite[Theorem~9.2]{Birman:Solomyak:1977}.
Note that the condition that the function $k$ is bounded cannot be removed,
as the space of $\mathcal L_2$-Schur multipliers is precisely the set of
bounded functions.

\begin{thm}
	\label{thm:schur}
	For any index $p\in(0,2)$, if
	$k\in\dot{B}^{\nicefrac{1}{p^\flat}}_{p^\flat,p}
	(\mathbb R,L_\infty(\mathbb R))\cap L_\infty(\mathbb R^2)$,
	then $k$ is an $\mathcal L_p$-Schur multiplier,
	with the quasinorm estimate
	\[
		\norm{k}_{\mathfrak M_p}\lesssim_p
			|k|_{\dot{B}^{\nicefrac{1}{p^\flat}}_{%
			p^\flat,p}(\mathbb R,L_\infty(\mathbb R))}
			+
			\norm{k}_{\infty}.
	\]
\end{thm}

\begin{proof}
	Given that $k$ is bounded,
	\[
		k(t,s)=c(t)+\sum\limits_{j\in\mathbb Z}
		\left( k_j(t,s)-k_j(t,0) \right),
	\]
	for all $s,t\in\mathbb R$, where $c$ is bounded, and satisfies
	\[
		\norm{c}_\infty\lesssim
		\norm{k}_\infty+
		|k|_{\dot{B}^s_{p,q}(\mathbb R, L_\infty(\mathbb R))}.
	\]
	By Corollary~\ref{crl:schur:est} and the $p$-triangle inequality,
	we have that
	\begin{align*}
		\norm{k}^p_{\mathfrak M_p}
		&\lesssim
		\norm{c}_\infty^p+
		\sum\limits_{j\in\mathbb Z}
		\left( 
			\norm{k_j}_{\mathfrak M_p}^p+
			\norm{k_j}^p_\infty
		\right)\\
		&\lesssim
		\norm{c}_\infty
		+
		\sum\limits_{j\in\mathbb Z}
		2^{\nicefrac{jp}{2}}
		\left( 
			\sum\limits_{l\in\mathbb Z}
			\norm{k_{j,l}}^{p^\flat}_\infty
		\right)^{\nicefrac{p}{p^\flat}}\\
		&\lesssim
		\norm{k}_\infty+
		|k|_{\dot{B}^{\nicefrac{1}{p^\flat}}_{p^\flat,p}
		(\mathbb R,L_\infty(\mathbb R))}
		+
		\sum\limits_{j\in\mathbb Z}
		2^{\nicefrac{jp}{2}}
		\left( 
			\sum\limits_{l\in\mathbb Z}
			\norm{k_{j,l}}^{p^\flat}_\infty
		\right)^{\nicefrac{p}{p^\flat}}\\
		&\lesssim
		\norm{k}_\infty+
		|k|_{\dot{B}^{\nicefrac{1}{p^\flat}}_{%
		p^\flat,p}(\mathbb R,L_\infty(\mathbb R))},
	\end{align*}
	where the last inequality follows by the definition of the
	$L_\infty$-valued Besov space.
\end{proof}

\begin{rmk}
    Let $D_x$ denote the unbounded self-adjoint operator on $L_2(\mathbb{R})$
	given by $D_x\xi(t) = -i\xi'(t).$
    It was observed by Birman and Solomyak that if $k$ is a Schur multiplier
	of $\mathcal{L}_1$, then the transformer $T^{M_x,D_x}_{k}$ acts boundedly in
	the operator norm, and
    \begin{equation*}
        T^{M_x,D_x}_k(1)
    \end{equation*}
    coincides with a pseudodifferential operator with symbol function $k$, 
	see \cite[Section 6]{Birman:Solomyak:2003}. 
	Theorem \ref{thm:schur} therefore gives a new proof of the result that a
	pseudodifferential operator with symbol function bounded and belonging to
	$\dot{B}^{\nicefrac{1}{2}}_{2,1}(\mathbb{R},L_\infty(\mathbb{R}))$ defines a
	bounded linear operator on $L_2(\mathbb{R})$
	(see \cite[Equation~(6.3)]{Birman:Solomyak:2003} for a weaker example of such
	a result).
    
    This result should also be compared to the work of Sugimoto for similar estimates on pseudo-differential
	operators \cite{Sugimoto:1988b}.
\end{rmk}
%

\end{document}